\documentclass[12pt]{article}

\usepackage{fullpage,amsmath,amssymb,amsthm,xspace,url,verbatim,tikz,ifthen}

\newcommand{\newprob}[1]{\theoremstyle{definition}\newtheorem{#1}[thmctr]{Problem}\theoremstyle{plain}}

\DeclareMathOperator{\rt}{\mathbf{RT}}
\DeclareMathOperator{\diam}{\text{diam}}

\DeclareMathOperator{\ex}{\text{ex}}
\DeclareMathOperator{\tkop}{\mathrm{TK}}
\DeclareMathOperator{\tkfop}{\mathcal{TK}}

\newcommand{\diamt}[1]{d_\text{max}(#1)}
\newcommand{\tk}[2]{\tkop^{#2}(#1)}
\newcommand{\tkf}[2]{\tkfop^{#2}(#1)}
\newcommand{\HH}{\mathcal{H}}

\newboolean{includesphere}
\setboolean{includesphere}{false}

\newcommand{\joziuni}{University of Illinois \\ University of California, San Diego \\ jobal@math.uiuc.edu}
\newcommand{\jozithanks}{This material is based upon work supported by NSF CAREER Grant DMS-0745185,
 UIUC Campus Research Board Grants 09072 and 08086, and OTKA Grant K76099.}
\newcommand{\johnuni}{University of Illinois \\ jlenz2@math.uiuc.edu}
\newcommand{\johnthanks}{Work supported by 2010 REGS Program of the University of Illinois and the
 National Science Foundation through a fellowship funded by the grant DMS  0838434 ``EMSW21MCTP: Research Experience for Graduate Students``.}

\begin{document}

\title{On the Ramsey-Tur\'{a}n numbers of graphs and hypergraphs}
\date{\today}
\author{J\'{o}zsef Balogh \thanks{\jozithanks} \\ \joziuni  \and
       John Lenz \thanks{\johnthanks} \\ \johnuni }

\maketitle

\begin{abstract}
Let $t$ be an integer, $f(n)$ a function, and $H$ a graph. Define the $t$-Ramsey-Tur\'{a}n number of $H$,
$\rt_t(n, H, f(n))$, to be
the maximum number of edges in an $n$-vertex, $H$-free graph $G$ with $\alpha_t(G) \leq f(n)$, where
$\alpha_t(G)$ is the maximum number of vertices in a $K_t$-free induced subgraph of $G$.
Erd\H{o}s, Hajnal, Simonovits, S\'{o}s, and Szemer\'{e}di~\cite{rt-erdos94} posed several open questions
about $\rt_t(n,K_s,o(n))$, among them finding the minimum $\ell$ such that
$\rt_t(n,K_{t+\ell},o(n)) = \Omega(n^2)$, where it is easy to see
that $\rt_t(n,K_{t+1},o(n)) = o(n^2)$.  In this paper, we answer this question by proving that
$\rt_t(n,K_{t+2},o(n)) = \Omega(n^2)$;
our constructions  also
imply several results on the Ramsey-Tur\'{a}n numbers of hypergraphs.
\end{abstract}

\section{Introduction}
Let $\mathcal{H}$ be an $r$-uniform hypergraph and $f(n)$ a function.  The 
\textbf{Ramsey-Tur\'{a}n number of $\mathcal{H}$}, $\rt(n,\mathcal{H},f(n))$, is the
maximum number of edges in an $n$-vertex, $r$-uniform, $\mathcal{H}$-free hypergraph with independence number at most $f(n)$.
In 1970, Erd\H{o}s and S\'{o}s~\cite{rt-erdos70} initiated the study of Ramsey-Tur\'{a}n numbers of graphs
when they started investigating
whether excluding  large independent sets in $K_{s}$-free graphs implies an improvement in Tur\'an's theorem. 
One of the main problems in Ramsey-Tur\'{a}n theory is to determine the threshold function for $\mathcal{H}$ (see~\cite{rt-simonovits01} for a survey).
The \textbf{threshold function} for $\mathcal{H}$ is a function $t(n)$ such that $\rt(n,\mathcal{H},t(n)) = \Omega(n^r)$ and if $f(n) = o(t(n))$ then
$\rt(n,\mathcal{H},f(n)) = o(n^r)$.  Define
\begin{align*}
  \theta(\mathcal{H}) = \lim_{\epsilon \rightarrow 0} \lim_{n \rightarrow \infty} \frac{\rt(n,\mathcal{H},\epsilon n)}{n^r}.
\end{align*}
In an abuse of notation, we
write
$\rt(n,\mathcal{H},o(n)) = \theta(\mathcal{H}) n^r + o(n^r)$.  An easy diagonalization argument shows that $t(n) = n$ is a
threshold function for $\mathcal{H}$ if and only if $\ex(n,\mathcal{H}) = \Omega(n^r)$ and $\theta(\mathcal{H}) = 0$.
Very few threshold functions are known exactly; instead we study the easier problem of deciding whether $t(n) = n$ is
a threshold function or not.

Erd\H{o}s, Hajnal, S\'{o}s, and Szemer\'{e}di~\cite[p. 80]{rt-erdos83} proposed a problem
about an extension of the concept of the Ramsey-Tur\'{a}n numbers of graphs.
Let $G$ be a graph and define the \textbf{$K_t$-independence number} of $G$ as  $$\alpha_t(G):= 
\max \left\{ \left| S \right| : S \subseteq V(G), G[S] \text{ is } K_t \text{-free} \right\}.$$  
Define $\rt_t(n,H,f(n))$ to be the maximum number of edges in an $H$-free graph $G$ on $n$ vertices with $\alpha_t(G) \leq f(n)$ and define
\begin{align} \label{rtdef}  
 \theta_t(\mathcal{H})=  \lim_{\epsilon \rightarrow 0} \lim_{n \rightarrow \infty} \frac{1}{n^2} \rt_t(n,H,\epsilon n).
\end{align}

We write $\rt_t(n,H,o(n))= \theta_t(\mathcal{H}) n^2 + o(n^2)$.
For $t = 2$, it is easy to show that the limit in \eqref{rtdef} exists; for $t \geq 3$, the fact that
these limits exist is not obvious, it was one of the main results in \cite{rt-erdos94}.

For complete graphs of odd order, Erd\H{o}s and S\'{o}s~\cite{rt-erdos70} proved that
\begin{align*}
\theta( K_{2s+1}) = \frac{1}{2} \left(1-\frac{1}{s}\right),
\end{align*}
leaving open the question of determining $\theta( K_{2s})$ for $s\ge 2$.
The first celebrated result in Ramsey-Tur\'{a}n theory was that $\theta( K_4) = \frac{1}{8}$.
In one of the first applications of the Regularity Lemma\footnote{This was an earlier version of the  Szemer\'edi Regularity Lemma.} to graph theory, Szemer\'{e}di~\cite{rt-szemeredi72}
proved that $\theta(K_4) \leq \frac{1}{8}$ in 1972.  Four years later,
Bollob\'{a}s and Erd\H{o}s~\cite{beg-bollobas76} provided a surprising geometric construction
using high dimensional spheres which
proved that Szemer\'{e}di's upper bound was tight.
It was not until 1983 that Erd\H{o}s, Hajnal, S\'{o}s, and Szemer\'{e}di~\cite{rt-erdos83}
 extended this result to all complete graphs of even order, determining
$\theta(K_{2s}).$ 

It was also proved in  \cite{rt-erdos83} that $\theta(H) \leq \theta(K_s)$ for $s\ge 5$,
 where $s$ is the minimum integer for which $V(H)$ can be partitioned into $\left\lceil s/2 \right\rceil$ sets  $V_1, \ldots, V_{\left\lceil s/2 \right\rceil}$
such that $V_1, \ldots, V_{\left\lfloor s/2 \right\rfloor}$ span forests in $H$ and if $s$ is odd then  $V_{\left\lceil s/2 \right\rceil}$ spans an independent set.
For odd $s $ this bound is sharp. The `simplest' major   open question is to decide if $\theta(K_{2,2,2})=0$.

The exact threshold function for $K_s$ for $s \geq 4$ is also still unknown, but Sudakov~\cite{rt-sudakov03}
showed, using the so called ``dependent random choice method,'' that $\rt(n,K_4,n2^{-\omega \sqrt{\log n}}) = o(n^2)$, where $\omega = \omega(n)$ is any function
going to infinity arbitrarily slowly.  Note that $n2^{-\omega \sqrt{\log n}}/n^{1-\delta} \rightarrow \infty$
as $n\rightarrow \infty$ for any fixed $\delta$.

No results about $\rt_t(n,H,o(n))$ for $t \geq 3$ were known until  Erd\H{o}s, Hajnal, Simonovits, S\'{o}s,
and Szemer\'{e}di~\cite{rt-erdos94} proved that  the limit in \eqref{rtdef} exists when $H$ is a complete graph,
$\theta_t(K_s) \leq \frac{1}{2} \left( 1 - \frac{t}{s-1} \right) $, and this 
is sharp for all $s \equiv 1 \pmod t$.  Note that for $t = 2$ this was already known by Erd\H{o}s and S\'{o}s~\cite{rt-erdos70}.
Additionally, for some special cases, for $\ell = 1,2,3,4, 5$ and $\ell \leq t+1$ they proved that $$\theta_t(K_{t+\ell}) \leq \frac{\ell - 1}{4t} .$$
In~\cite{rt-erdos94} a
 construction was given  proving that 
\begin{equation}\label{2t_t8}
\theta_t(n,K_{2t},o(n)) \geq \frac{1}{8}.
\end{equation}
Unfortunately, the proof that the constructed graph
has small independence number relied on a theorem of Bollob\'{a}s~\cite{beg-bollobas89}
which has been withdrawn as incorrect~\cite{beg-bollobas10}.
Therefore, until now, it was unknown if  $\theta_t(K_s)$ is positive for   $s \leq 2t$.
Erd\H{o}s, Hajnal, Simonovits, S\'{o}s,
and Szemer\'{e}di~\cite{rt-erdos94} posed several open problems.

\newprob{pminl}
\begin{pminl}\label{pminl}
(\cite[Problem 2.12]{rt-erdos94}) Find the minimum $\ell$ such that $\theta_t(K_{t+\ell})  >0$.
\end{pminl}

In~\cite{rt-erdos83}, Erd\H{o}s, Hajnal, S\'{o}s, and Szemer\'{e}di wrote that to solve Problem~\ref{pRT35} below ``an analogue of the Bollob\'{a}s-Erd\H{o}s
graph would be needed which we think will be extremely hard to find.''

\newprob{pRT35}
\begin{pRT35} \label{pRT35}
(\cite{rt-erdos94},~\cite{rt-erdos83}, and~\cite[Problem 17]{rt-simonovits01}) 
Determine if $\theta_3(K_5) >0$.
\end{pRT35}

As we already mentioned,  Problem~\ref{pRT35} is motivated by the history of the analogous question for
 $t=2$: 
Erd\H{o}s and S\'{o}s~\cite{rt-erdos70} observed that $\theta(K_5) > 0$ and
$\theta(K_3) = 0$, leaving open the hard problem deciding if $\theta(K_4)>0$, which was solved by   Bollob\'as
and Erd\H os~\cite{beg-bollobas76}.   For $ t = 3$, it is easy to observe  
 that $\theta_3(K_4) = 0$ and $\theta_3(K_7) > 0$, motivating Problem~\ref{pRT35}.

\section{Results}

The main result of our paper is solving Problems~\ref{pminl} and \ref{pRT35} by constructing 
graphs showing that  $\theta_t(K_{t+\ell})>0$ for $2 \leq \ell \leq t$. This is a breakthrough step in the area;
 in this part of extremal graph theory constructions usually do not come easily.

\newtheorem{mainind3k5}[thmctr]{Theorem}
\begin{mainind3k5} \label{mainind3k5}
For $t \geq 2$ and $2 \leq \ell \leq t$, let $u =  \left\lceil  t/2 \right\rceil$.  Then
\begin{align*}
  \theta_t(K_{t+\ell}) \geq \frac{1}{2} \left( 1 - \frac{1}{\ell} \right) 2^{-u^2}.
\end{align*}
\end{mainind3k5}

For comparison, trivially, $\theta_t(K_{t+1}) = 0$. Note that for $t=\ell$ the bound in \eqref{2t_t8} is better than in Theorem~\ref{mainind3k5}.
 Theorem~\ref{mainind3k5} can also  be used to give a lower bound for  $\theta_t(K_{qt+\ell})$ for all  $q \geq 1$ and $2 \leq \ell \leq t$.
Let $G$ be a member of the graph sequence 
 constructed to prove Theorem~\ref{mainind3k5} and let $T$ be a complete $(q-1)$-partite
graph with almost equal class sizes.
In each class of $T$, insert a $K_{t+1}$-free graph with small $K_t$-independence number.  (Such a graph
exists by the Erd\H{o}s-Rogers Theorem~\cite{rt-erdos62}.)  Lastly, completely join $G$ and $T$.  Any copy of $K_{qt+\ell}$
which appears
in this graph can have at most $t$ vertices in each part of $T$.  This forces $G$ to contain $t+\ell$ vertices of the copy of
$K_{qt+\ell}$, which is a contradiction. This graph also has a small $K_t$-independence number.  
Letting $|G|=an$ we instantly conclude the following.

\newtheorem{corindks}[thmctr]{Corollary}
\begin{corindks} \label{corindks}
For $t,q \geq 2$ and $2 \leq \ell \leq t$, let $u = \left\lceil  t/2 \right\rceil$.  For any $0 < a < 1$,
\begin{equation}
\theta_t(K_{qt+\ell}) \geq \frac{1}{2} \left( 1- \frac{1}{\ell} \right) 2^{-u^2} a^2 + \binom{q-1}{2}\left( \frac{1-a}{q-1} \right)^2 + (1-a)a.
\end{equation}
\end{corindks}
The precise
formula for $a$ is cumbersome so it is not included here.  Instead, we list
the optimized value for some small values of $s$ and $t$.  In ~\cite{rt-erdos94} the following problem was posed.
\newprob{pRT3small}
\begin{pRT3small} \label{pRT3small}
(\cite{rt-erdos94},~\cite{rt-erdos83}, and~\cite[Problem 19]{rt-simonovits01}) 

 $$\theta_3(K_5) \leq \frac{1}{12},\ \
 \theta_3(K_6) \leq \frac{1}{6},\ \
 \theta_3(K_8) \leq \frac{3}{11},\ \
 \theta_3(K_9) \leq \frac{3}{10} .$$
Are any of these bounds tight?
\end{pRT3small}

 Optimizing $a$  in Corollary~\ref{corindks},  the relative
size of $G$ and $T$,  we obtain the following lower bounds for the graphs
considered in Problem~\ref{pRT3small}.
$$\frac{1}{64} \le \theta_3(K_5),\ \
\frac{1}{48} \le \theta_3(K_6),\ \
\frac{16}{63}\le \theta_3(K_8),\ \ 
 \frac{12}{47}\le \theta_3(K_9).$$  

Note that in \cite{rt-erdos94} it was proved that $\rt_3(n,K_7,o(n)) = \frac{1}{4} n^2 + o(n^2)$.

 Theorem~\ref{mainind3k5} follows from a result about hypergraphs.
For $s>r$ let $\tk{s}{r}$ be the $r$-uniform hypergraph obtained from the complete graph $K_s$ by replacing each graph edge $uv$ with a hypergraph edge which besides $u,v$ contains $r-2$ new vertices.  The \textbf{core vertices}
of $\tk{s}{r}$ are the $s$ vertices of degree larger than one.
Let $\tkf{s}{r}$ be the family of $r$-uniform hypergraphs $\HH$ such that there exists a set $S$ of $s$ vertices of $\HH$ where
each pair of vertices from $S$ are contained  in some hyperedge of $\HH$.  The set $S$ is called the set of \textbf{core vertices} of  $\HH$.

Let $T_s^r(n)$ be the complete $n$-vertex, $r$-uniform, $s$-partite hypergraph with part sizes as equal as possible.
Mubayi~\cite{rt-mubayi06-2} showed for $s>r$ that  $\ex(n, \tkf{s+1}{r}) = \left| T_s^r(n) \right|$ and
$\ex(n,\tk{s+1}{r}) = (1 + o(1)) \left| T_s^r(n) \right|$.  Recently, Pikhurko~\cite{rt-pikhurko},  improving on  ~\cite{rt-mubayi06-2},  has shown that for large $n$,
$\ex(n,\tk{s+1}{r}) = \left| T_s^r(n) \right|$ and that $T_s^r(n)$ is the unique extremal example.  Since in this case the extremal
hypergraphs have large independent sets, it is interesting to study the behavior of the function $\rt(n,\tk{s}{r},f(n))$ for $f(n)=o(n)$.
A simple observation is the following.

\newtheorem{tkrprop}[thmctr]{Proposition}
\begin{tkrprop}
For $r \geq 2$,
\begin{enumerate}
\item $\rt(n,\tk{r+1}{r}, o(n)) = o(n^r)$.
\item $\rt(n,\tkf{2r-1}{r},o(n)) = o(n^r)$.
\end{enumerate}
\end{tkrprop}

\begin{proof}
We  prove only the statement  for $3$-uniform hypergraphs; the proof can be easily extended to every $r\ge 3$.

Let $\mathcal{H}$ be a $3$-uniform, $n$-vertex hypergraph with independence number at most $\epsilon n$ and at least $9\epsilon n^3 + 72n^{2}$ edges.
For simplicity, assume $3$ divides $n$ and let $\mathcal{H}'$ be a $3$-partite subhypergraph of $\mathcal{H}$ with equal part sizes and with at least $\frac{1}{9}$ of the edges of $\mathcal{H}$.  
Recall that for a pair of vertices $x$ and $y$, their \textbf{codegree} $d(x,y)$ 
is the number of vertices $z$ such that $\left\{ x,y,z \right\}$ is an edge.
For each pair  $x,y$ of vertices  in different classes, delete
all edges containing $x$ and $y$ if their codegree is at most $16$.  We delete at most $8n^{2}$ hyperedges.
Thus we have a $3$-partite hypergraph $\mathcal{H}'$ with at least $\epsilon n^3$ edges and the
codegree of any pair of vertices from different classes is zero or at least $16$.

Since $\mathcal{H}'$ has at least $\epsilon n^3$ cross-edges, the maximum codegree of $\mathcal{H}'$ is at least $\epsilon n$. Let
$x,y$ be a pair of vertices from different classes with codegree at least $\epsilon n$, and let $Z$ be the set of vertices $z$ in the third class
such that $\left\{ x, y, z \right\}$ is an edge.  Since the independence number of $\mathcal{H}$ is at most $\epsilon n$,
there exists a hyperedge $E$ of $\mathcal{H}$ contained in $Z$.  The vertices in $E$ together with $x,y$ form a hypergraph
in $\tkf{5}{3}$. (In the $r$-uniform case, the edge $E$ together with $r-1$ vertices will form a copy of $\tkf{2r-1}{r}$.)
Thus any $3$-uniform, $n$-vertex, $\tkf{5}{3}$-free hypergraph with independence number at most $\epsilon n$
can have at most $9\epsilon n^3 + 72 n^2$ edges.

To find a copy of $\tk{4}{3}$, let $z_1$ and $z_2$ be two vertices from $E$.  The core vertices in a copy of $\tk{4}{3}$ are 
$x$, $y$, $z_1$, and $z_2$.  The vertices $z_1$ and $z_2$ are contained together in the edge $E$, and since $x$ and $z_1$ are contained
together in a hyperedge of $\mathcal{H}'$, the codegree of $x$ and $z_1$ is at least $16$.  Thus we can find an edge of $\mathcal{H}$
containing $x$ and $z_1$ where the third vertex avoids all previously used vertices.  Similarly we can find edges containing $x,z_2$ and $x,y$ and $y,z_i$ where the third vertex
has not yet been used.  Thus we find a copy of $\tk{4}{3}$ in $\mathcal{H}$.  (In the $r$-uniform case,
take as core vertices two vertices from $E$ together with $r-1$ other vertices to find a copy of $\tk{r+1}{r}$.)
\end{proof}

Using our construction, we prove the following lower bounds.

\newtheorem{maintkthm}[thmctr]{Theorem}
\begin{maintkthm} \label{maintkthm}
Let $r \geq 3$ and let $u= \left\lceil  r/2 \right\rceil$.
\begin{itemize}
   \item[(i)] $\rt(n,\tk{r+2}{r}, o(n)) \geq 2^{-\binom{ur}{2} + r\binom{u}{2}} \left( \frac{n}{r} \right)^r$.
   \item[(ii)] $\rt(n,\tkf{2r}{r}, o(n)) \geq 2^{-\binom{ur}{2} + r\binom{u}{2}} \left( \frac{n}{r} \right)^r$.
\end{itemize}
\end{maintkthm}

\newcommand{\refmaintkthm}{\ref{maintkthm}  (i)}
\newcommand{\refmaintkfthm}{\ref{maintkthm}  (ii)}

Note that unlike in the Tur\'an-density  extremal case, where for large $n$ we have $\ex(n,\tk{s}{r}) = \ex(n,\tkf{s}{r})$, the
Ramsey-Tur\'{a}n numbers for $\tk{s}{r}$ and $\tkf{s}{r}$ are different.  Let $\mathcal{F}^{r}(s)$ be the subfamily
of $\tkf{s}{r}$ containing those hypergraphs where each edge contains exactly two core vertices.
We can prove similarly to Theorem~{\refmaintkthm} that $\theta(\mathcal{F}^{r}({r+2})) \geq 2^{-\binom{ur}{2} + r\binom{u}{2}} \left( \frac{1}{r} \right)^r$. Based on this 
  we conjecture 
that $\mathcal{F}^{r}(s)$  behaves like $\tk{s}{r}$.  

\newtheorem{tkfconj}[thmctr]{Conjecture}
\begin{tkfconj}
$\theta(\mathcal{F}^{r}(s)) = \theta(\tk{s}{r})$ for $s>r$.
\end{tkfconj}

Theorems~\ref{mainind3k5} and \ref{maintkthm} are corollaries of the following theorem, which is our main tool.

\newtheorem{constrblowup}[thmctr]{Theorem}
\begin{constrblowup} \label{constrblowup}(Construction)
For any integer $r \geq 2$ there exist positive constants $c_1,c_2$ such that the following holds. For arbitrary small constants $\alpha, \beta > 0$, and any integer $N$, there exists an $m \geq N$  such that there exists
an $rm$-vertex, $r$-uniform hypergraph $\mathcal{G}$ with vertex partition $W_1, \ldots, W_r$ with $\left| W_i \right| = m$ and the
following properties:
\begin{itemize}
\item[(i)]  No subhypergraph  from $\tkf{4}{r}$ is embedded into 
$\mathcal{G}$  so that both $W_i$ and $W_j$ contain two core vertices for some $i \neq j$.
\item[(ii)] $ e(\mathcal{G})  \geq 2^{r\binom{u}{2} - \binom{ru}{2}} m^r - c_2\alpha m^r$, where $u = \left\lceil  r/2 \right\rceil$.
\item[(iii)] For any $i$, $\mathcal{G}[W_i]$ contains  no connected  hypergraph
       $\mathcal{F}$ with $\left| V(\mathcal{F}) \right| \leq r^3$ and
$$\left| V(\mathcal{F}) \right| < r + (r-1)(\left| \mathcal{F} \right|-1).$$
\item[(iv)] The independence number of $\mathcal{G}$ is at most $c_1\beta m$.
\end{itemize}
\end{constrblowup}

We show that if  the independence number
of a $\tk{6}{3}$-free  $3$-uniform hypergraph with $n$ vertices is at most $n2^{-\omega (\log n)^{2/3}}$,
then  it has $o(n^3)$ edges.  The proof of Theorem~\ref{k3tkthm} for every $r \geq 3$ extends to  $\tk{2r}{r}$-free $r$-uniform hypergraphs  with independence number at most
$n2^{-\omega (\log n)^{(r-1)/r}}$,  we omit the details.

\newtheorem{k3tkthm}[thmctr]{Theorem}
\begin{k3tkthm} \label{k3tkthm}
Let $w = w(n)$ be any function tending to infinity arbitrarily slowly, and let
$f(n) = n2^{-w\cdot (\log n)^{2/3}}$.  Then $\rt(n,\tk{6}{3}, f(n)) = o(n^3)$.
\end{k3tkthm}

In Section~\ref{secSphereProp} we state several properties of the $k$-dimensional
unit sphere which will be used  in the construction.  In Section~\ref{secFormer} we describe two
earlier constructions by Bollob\'{a}s and Erd\H{o}s \cite{beg-bollobas76}
and R\"{o}dl \cite{beg-rodl85}.
In Section~\ref{secConstruction} we
describe our construction and prove several properties of it, and in Section~\ref{secTheorems} we
show how the construction presented in Section~\ref{secConstruction} can be modified to prove
Theorems~\ref{mainind3k5}, \ref{maintkthm}, and \ref{constrblowup}.  In
Section~\ref{secSmallInd} we prove Theorem~\ref{k3tkthm} and lastly we state some open problems
in Section~\ref{secOpen}.
Throughout the paper, we often omit the floor and ceiling signs for the sake of simplicity.


\section{Properties of the unit sphere} \label{secSphereProp}

Let $\mu$ be the Lebesgue measure on the $k$-dimensional unit sphere $\mathbb{S}^k \subseteq \mathbb{R}^{k+1}$ normalized so that $\mu(\mathbb{S}^k) = 1$.
For $A \subseteq \mathbb{S}^k$, define $\diam(A) = \sup \left\{ d(x,y) : x, y \in A \right\}$ where $d(x,y)$ is the Euclidean distance in $\mathbb{R}^{k+1}$.
For $A,B \subseteq \mathbb{S}^k$, define $$\diamt{A,B}  = \sup \{ d(a,b) : a \in A, b \in B \}.$$
For $A \subseteq \mathbb{S}^k$ and $t \geq 2$, define
\begin{align*}
d_t(A) = \sup \left\{ \min_{i \neq j} d(x_i, x_j) : x_1, \ldots, x_t \in A \right\}.
\end{align*}
In particular,   let $\delta = \delta_t$ be the
edge length of the $t$-simplex, i.e., $\delta_t = \sqrt{\frac{2t}{t-1}}$.
A \textbf{spherical cap} is the intersection of the unit sphere $\mathbb{S}^k$ with a halfspace.
The \textbf{center} of a spherical cap is the point in the spherical cap at maximum distance from $H$, where $H$ is the hyperplane bounding the halfspace.
The \textbf{height} of a spherical cap is the minimum distance between the center and $H$ and
the \textbf{diameter} of a spherical cap is the diameter of the sphere formed by the intersection of the spherical cap with $H$.
Note that if $a$ is the maximum distance between the center and a point of the spherical cap and $h$ is the height, then $2h = a^2$.

\medskip

\newcommand{\propsize}{(P1)\xspace}
\newcommand{\propsizemulti}{(P2)\xspace}
\newcommand{\propcap}{(P3)\xspace}
\newcommand{\propbad}{(P4)\xspace}
\newcommand{\propdiam}{(P5)\xspace}
\newcommand{\propdiamt}{(P6)\xspace}

\newcommand{\propsizenospace}{(P1)}
\newcommand{\propsizemultinospace}{(P2)}
\newcommand{\propcapnospace}{(P3)}
\newcommand{\propbadnospace}{(P4)}
\newcommand{\propdiamnospace}{(P5)}
\newcommand{\propdiamtnospace}{(P6)}
\newcommand{\propdiamexpand}{(P7)}
\newcommand{\propdiamtehsssz}{(P8)}
\newcommand{\propdiamexpandb}{(P7)\xspace}
\newcommand{\propdiamtehssszb}{(P8)\xspace}
\newcommand{\propdiambip}{(P9)}
\newcommand{\propdiambipb}{(P9)\xspace}

Given any $\alpha, \beta > 0$, it is possible to select $\epsilon > 0$ small enough and then $k$ large enough so that Properties~\propsize, \propsizemulti, and \propcap below
are satisfied.

\begin{itemize}
\item[(P1)] Let $C$ be a spherical cap in $\mathbb{S}^k$ with height $h$, where $2h = \left( \sqrt{2} - \epsilon/\sqrt{k} \right)^2$
 (this means that all points of the spherical cap are within distance $\sqrt{2} - \epsilon/\sqrt{k}$ of the center).
  Then $\mu(C) \geq \frac{1}{2} - \alpha$.
\item[(P2)] Let $C_1, \ldots, C_t$ be spherical caps in $\mathbb{S}^k$ with height $h$, where $2h = \left( \sqrt{2} - \epsilon/\sqrt{k} \right)^2$.
  Let $z_i$ be the center of $C_i$.  Assume for all $1 \leq i < j \leq t$ that $d(z_i, z_j) \leq \sqrt{2}$.  Then $\mu(C_1 \cap \ldots \cap C_t) \geq \frac{1}{2^t} -t\alpha$.
\item[(P3)] Let $C$ be a spherical cap with diameter $2 - \epsilon/(2\sqrt{k})$.  Then $\mu(C) \leq \beta$.
\end{itemize}
We also use the following properties of high dimensional spheres.
\begin{itemize}
\item[(P4)] For any $0 < \gamma < \frac{1}{4}$, it is impossible to have $p_1, p_2, q_1, q_2 \in \mathbb{S}^k$ such that $d(p_1, p_2) \geq 2 - \gamma$,
$d(q_1, q_2) \geq 2 - \gamma$, and $d(p_i,q_j) \leq \sqrt{2} - \gamma$ for all $1 \leq i,j \leq 2$.
\item[(P5)] Let $A \subseteq \mathbb{S}^k$ and let $C$ be a spherical cap of the same measure.
Then $\diam(A) \geq \diam(C)$.
\item[(P6)] Let $A,B \subseteq \mathbb{S}^k$ with equal measure and let $C$ be a cap of the same measure.  Then $\diamt{A,B} \geq \diam(C)$.
\end{itemize}

Properties~\propsize and \propsizemulti follow directly from the formula for the measure of a spherical cap,
Properties~\propcap, \propdiam,   and \propdiamt are all folklore results that are easy corollaries of the isoperimetric inequality on
the sphere \cite{beg-leader10}, and Property~\propbad is from \cite{beg-bollobas76},
 see also ~\cite{rt-erdos94}.

Erd\H{o}s, Hajnal, Simonovits, S\'{o}s,
and Szemer\'{e}di~\cite{rt-erdos94} gave a construction which they claim
proved $\rt_t(n,K_{2t},\linebreak[1] o(n)) \geq \frac{1}{8}n^2-o(n^2)$.  Unfortunately, the proof that the construction
has small independence number relies on a theorem of Bollob\'{a}s~\cite{beg-bollobas89} which has been withdrawn
as incorrect~\cite{beg-bollobas10}.  
In \cite{beg-bollobas89}, the following question was considered.
Is it true that if $C$ is a spherical cap with $\mu(C) = \mu(A)$, then $d_t(A) \geq d_t(C)$?  If this were true
as claimed in \cite{rt-erdos94}, 
then $\theta_t(K_{2t}) \geq \frac{1}{8}$.  
In a private communication, Bollob\'{a}s~\cite{beg-bollobas10} provided the following counterexample.
Take $C$ to be a cap of the sphere in three dimensions with small but
positive measure and let $C'$ be another cap of the same measure which is far from $C$.  Let $A = C \cup C'$.
Then if $\mu(C)$ is small enough we can approximate $C$ and $C'$ by circles with radius $r$.
Then $d_3(A) \approx 2r$ since we can take two points of $C$ and one point of $C'$.  But if $D$ is a cap
with the same measure as $A$ then $D$ has radius about $\sqrt{2}r$ so $d_t(D) \approx \sqrt{6}r > d_3(A)$.
This counterexample can be extended to higher dimensions and more than three points, 
but only seems to work when $C$ has small measure.


\ifthenelse{\boolean{includesphere}}{
\input{sphere-proofs.tex}
}{}

\section{Former constructions} \label{secFormer}

In this section, we describe two previous constructions; our construction will use ideas from both.

\textbf{The Bollob\'{a}s-Erd\H{o}s Graph}, \cite{beg-bollobas76}.
In order to prove that $\rt(n,K_4,o(n)) \geq \frac{n^2}{8}-o(n^2)$, we need to construct, for every $\alpha,\beta > 0$,
a $K_4$-free graph $G$ with $n$ vertices, independence number at most $\beta n$, and at least $\frac{n^2}{8} (1 - \alpha)$ edges.
Given $\alpha, \beta \geq 0$, take $\epsilon$ small enough and $k$ large enough so that Properties~\propsize and \propcap hold.
Divide the $k$-dimensional unit sphere $\mathbb{S}^k$ into $n/2$ domains having equal measure and diameter at most $\frac{\epsilon}{10 \sqrt{k}}$.  Choose a point from each domain
and let $P$ be the set of these points.  Let $\phi : P \rightarrow \mathcal{P}(\mathbb{S}^k)$ map points of $P$ to the corresponding domain of the sphere.
Take as vertex set of $G$ the disjoint union of two sets $V_1$ and $V_2$ both isomorphic to $P$.
For $x,y \in V_i$ we make
$xy$ an edge of $G$ if $d(x,y) \geq 2 - \epsilon/\sqrt{k}$.  For $x \in V_1, y \in V_2$ we make $xy$ an edge of $G$ if $d(x,y) \leq \sqrt{2} - \epsilon/\sqrt{k}$.
Then Property~\propsize shows that every vertex in $V_1$ has at least $\frac{1}{2} \left| V_2 \right| (1 - \alpha)$
neighbors in $V_2$ so the total number of edges is at least $\frac{1}{8} n^2 (1 - \alpha)$.
If $I$ is a set in $V_1$ with $\left| I \right| \geq \beta \left| V_1 \right| = \beta \frac{n}{2}$,
then $\mu(\phi(I)) = \left| I \right|/\left| P \right| \geq \beta$.
Let $C$ be a spherical cap of measure $\mu(\phi(I))$.
Properties~\propcap and \propdiam show that $2 - \epsilon/(2\sqrt{k}) \leq \diam(C) \leq \diam(\phi(I))$.
For $p \in I$, each $\phi(p)$ has diameter at most $\epsilon/(10\sqrt{k})$ so we can find two points $p_1, p_2 \in I$ with $d(p_1, p_2) \geq 2 - \epsilon/\sqrt{k}$, showing
that $I$ is not independent.
Finally, Property~\propbad shows this graph has no $K_4$ as a subgraph since any $K_4$ must take two vertices from $V_1$
and two vertices from $V_2$ (the graph spanned by $V_i$ is triangle-free).
To summarize, we have constructed a $K_4$-free graph $G$ on $n$ vertices with independence number at most $\beta n$ and
at least $\frac{1}{8} n^2 ( 1- \alpha)$ edges.
Since this construction holds for any $\alpha,\beta > 0$, we have proved that $\theta(K_4) \ge \frac{1}{8}$.


\textbf{The R\"{o}dl Graph},
\cite{beg-rodl85}.
We do not know if $\rt(n,K_{2,2,2},o(n))$ is $\Omega(n^2)$ or not.
Erd\H{o}s suggested that perhaps some modified version of the Bollob\'{a}s-Erd\H{o}s graph could be used to show it is $\Omega(n^2)$.  In this direction,
R\"{o}dl showed how to modify the Bollob\'{a}s-Erd\H{o}s graph to exclude both $K_4$ and $K_{3,3,3}$, proving that
$\theta(\left\{ K_4, K_{3,3,3} \right\}) \geq \frac{1}{8}$.
The R\"{o}dl Graph is formed by blowing up the Bollob\'{a}s-Erd\H{o}s Graph so that each vertex is blown
up into an independent set of size $t$ and then randomly delete edges from inside each $V_i$
(see Theorem~\ref{randblowup}).  By randomly deleting edges inside each $V_i$, we can destroy (almost) all
short cycles while not changing the density between $V_1$ and $V_2$.  Since the original graph does not contain $K_4$, 
blowing up the graph will not produce any $K_4$'s.  Also, after destroying all short cycles, any graph which is not the union of a bipartite graph with a forest, such as $K_{3,3,3}$,
will not be a subgraph of the final graph. One can check that the independence number of the obtained graph has smaller order of magnitude than its number of vertices.

\section{Construction} \label{secConstruction}

Erd\H{o}s, Hajnal, Simonovits, S\'{o}s,
and Szemer\'{e}di~\cite{rt-erdos94} conjectured (see~\cite[Conjecture 2.9]{rt-erdos94} and~\cite[Conjecture 18]{rt-simonovits01}) that the asymptotically
extremal graphs for $\rt_t(n,K_s,o(n))$ with $s = tq + \ell$ ($1 \leq \ell \leq t$)
have the following structure.  Partition $n$ vertices into $q+1$ classes
$V_0, \ldots, V_q$.  For each pair $\left\{ i,j \right\} \neq \left\{ 0,1 \right\}$ we almost completely join $V_i$ to $V_j$ and between $V_0$ and
$V_1$ we place a graph with density $(\ell - 1)/t + o(1)$.  Lastly, inside each $V_i$ we insert $o(n^2)$ edges.
By optimizing the sizes of the $V_i$s, the number of edges in this graph will be approximately
\begin{align*}
\left( 1 - \frac{2t - \ell + 1}{q(2t - \ell + 1) - \ell + 1} \right) \binom{n}{2}.
\end{align*}
 It was  suggested in \cite{rt-erdos94} that some modified version of the Bollob\'{a}s-Erd\H{o}s graph should be used between $V_0$ and $V_1$, but it was not known
how to reduce the density of the Bollob\'{a}s-Erd\H{o}s graph while still maintaining some useful properties.
Our construction is a modified version of the Bollob\'{a}s-Erd\H{o}s graph where we are able to reduce the density to roughly $2^{- t^2}$.
Unfortunately this is too low to match the conjecture but still enough to give a $\Omega(n^2)$ lower bound on $\rt_t(n,K_s,o(n))$.

Our construction depends on four parameters: two integers $r$ and $z$ and two small constants $\alpha, \beta > 0$.
Fix an integer $r \geq 3$.
Given $\alpha, \beta > 0$, fix $\epsilon$ and $k$ so that Properties~\propsize,  \propsizemulti and \propcap  hold. 
Define $\theta = \epsilon/\sqrt{k}$ and $u = \left\lceil  r/2 \right\rceil$.

For sufficiently large integer $z$, partition the $k$-dimensional unit sphere $\mathbb{S}^k$ into $z$ domains having equal measures and diameter at most $\theta/4$.  Choose a point from each set
and let $P$ be the set of these points.  Let $\phi : P \rightarrow \mathcal{P}(\mathbb{S}^k)$ map points of $P$ to the corresponding domain of the sphere.
The vertices of our hypergraph will be $r$ copies of ordered $u$-tuples of points from $P$.  Define
\begin{align*}
V &= \left\{ (p_1, \ldots, p_u) : p_i \in P \text{ and } d(p_i,p_j) \leq \sqrt{2} - \theta \text{ for all } i \neq j \right\}.
\end{align*}
We will denote vertices in $V$ as $\vec{v}$ and $\left<v^{(1)}, \ldots, v^{(u)}\right>$ as the coordinates of $\vec{v}$.
Let $V_1, \ldots, V_r$ be distinct sets isomorphic to $V$.
Let $\mathcal{H} = \mathcal{H}(r, z, \alpha,\beta)$ be the hypergraph with vertex set 
$V_1 \dot{\cup} \ldots \dot{\cup} V_r$ and the following hyperedges.

For each $1 \leq i \leq r$, make $E = \left\{\vec{v}_1, \ldots, \vec{v}_r\right\} \subseteq V_i$ a hyperedge if 
$\left| E \right| = r$ and for every pair $\vec{v}_{\ell}, \vec{v}_m$ (with $\ell \neq m$),
there exists some coordinate $1 \leq j \leq u$ such that
$d(\vec{v}_{\ell}^{(j)}, \vec{v}_m^{(j)}) \geq 2 - \theta$.
For cross-hyperedges, make $\left\{ \vec{v_1}, \ldots, \vec{v_r} \right\} \subseteq V(\mathcal{H})$ a hyperedge
if $\vec{v}_1 \in V_1, \ldots, \vec{v}_r \in V_r$ and
$d(v_i^{(j)}, v_{\ell}^{(m)}) \leq \sqrt{2} - \theta$ for all $1 \leq i, \ell \leq r$ and all $1 \leq j, m \leq u$.

First, we  claim some properties of $\mathcal{H}$.

\newtheorem{constrnotk4lem}[thmctr]{Lemma}
\begin{constrnotk4lem} \label{constrnotk4lem}
$\mathcal{H}$  contains no hypergraph in $\tkf{4}{r}$ embedded so that $V_i$ contains two core vertices and $V_j$ contains two core vertices with some $i \neq j$.
\end{constrnotk4lem}

\begin{proof}
Assume without loss of generality that there exist $\vec{v}_1, \vec{v}_2 \in V_1$ and $\vec{v}_3, \vec{v}_4 \in V_2$ which are all core vertices.
We will find four points violating Property~\propbad.
By the definition of hyperedges inside $V_1$, there is some coordinate $i$ such that $d(v_1^{(i)}, v_2^{(i)}) \geq 2 - \theta$.  Similarly,
there is some coordinate $j$ such that $d(v_3^{(j)}, v_4^{(j)}) \geq 2 - \theta$.  By the definition of cross-hyperedges, we know that all cross distances
are at most $\sqrt{2} - \theta$.  We therefore obtain four points $v_1^{(i)}, v_2^{(i)}, v_3^{(j)}, v_4^{(j)}$ which contradict Property~\propbad.
\end{proof}

\newtheorem{constrmatchinglem}[thmctr]{Lemma}
\begin{constrmatchinglem} \label{constrmatchinglem}
Let $A_1, A_2 \subseteq P$ with $\left| A_1 \right| = \left| A_2 \right| \geq 2 \beta z$ and let $t = \left| A_1 \right|/2$.
Then there exist $t$ distinct points $p_1, \ldots, p_t \in A_1$ and $t$ distinct points $q_1, \ldots, q_t \in A_2$ such
that $d(p_i, q_i) \geq 2 - \theta$.
\end{constrmatchinglem}

\begin{proof}
Let $G$ be the auxiliary bipartite graph on vertex set $A_1 \dot{\cup} A_2$ where $p \in A_1$ and $q \in A_2$ are adjacent if
$d(p,q) \geq 2 - \theta$.  We would like to find a matching of size at least $t$ in $G$.
Let $M$ be a maximum matching in $G$, and assume $\left| E(M) \right| < t$.  Let $G' = G - V(M)$ with $A_1' = A_1 - V(M)$
and $A_2' = A_2 - V(M)$.
We will show that $G'$ does not span an independent set, contradicting that $M$ is a maximum matching.

Since $\left| E(M) \right| < t$, $\left| A_i' \right| > t \geq \beta z$.
Let $B_i = \phi(A_i')$ so that $\mu(B_i) = \left| A_i' \right|/z \geq \beta$.
Let $C$ be a spherical cap of measure $\beta$ so $\mu(B_i) \geq \mu(C)$.
Properties~\propcap and \propdiamt show that $2 - \theta/2 \leq \diam(C) \leq \diamt{B_1, B_2}$.
Since each $\phi(p)$ has diameter at most $\theta/4$, we must have some $p \in A_1'$ and $q \in A_2'$ with $d(p, q) \geq 2 - \theta$.
In other words, $pq$ is an edge of $G'$ which contradicts that $M$ was a maximum matching.
\end{proof}

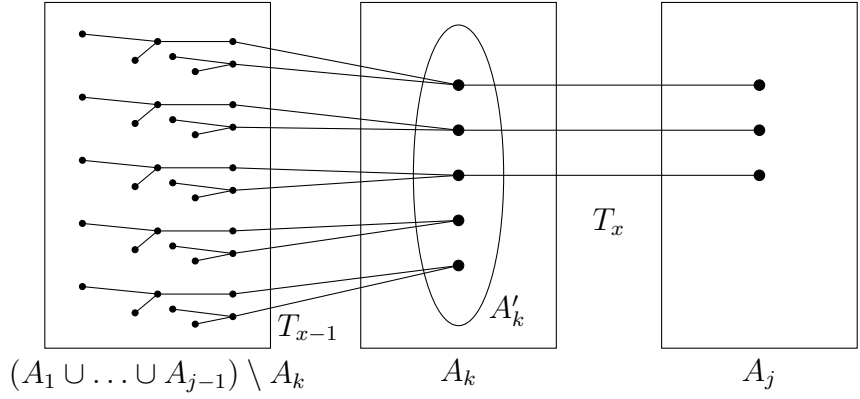
\begin{figure}
\begin{center}
\begin{tikzpicture}
\draw (-1.5,-2.3) rectangle (1.5,2.3);
\draw (0,-2.3) node[below] {$(A_1 \cup \ldots \cup A_{j-1}) \setminus A_k$};
\begin{scope}[xshift=4cm]
  \draw (-1.3,-2.3) rectangle (1.3,2.3);
  \draw (0,0) ellipse (0.6 and 2);
  \draw (0,-2.3) node[below] {$A_k$};
  \draw (0.65,-1.8) node {$A_k'$};
\end{scope}
\begin{scope}[xshift=8cm]
  \draw (-1.3,-2.3) rectangle (1.3,2.3);
  \draw (0,-2.3) node[below] {$A_j$};
\end{scope}
\foreach \y in {1.2,0.6,...,-1.2}{
\begin{scope}[yshift=\y*1.4 cm]
\draw (-1,0.2) -- (0,0.1);
\draw (-0.3,-0.15) -- (0,0.1) -- (1,0.1);
\draw (0.2,-0.1) -- (1,-0.2);
\draw (0.5,-0.3) -- (1,-0.2);
\draw[fill] (-1,0.2) circle (0.04cm);
\draw[fill] (-0.3,-0.15) circle (0.04cm);
\draw[fill] (0,0.1) circle (0.04cm);
\draw[fill] (1,0.1) circle (0.04cm);
\draw[fill] (0.2,-0.1) circle (0.04cm);
\draw[fill] (0.5,-0.3) circle (0.04cm);
\draw[fill] (1,-0.2) circle (0.04cm);
\end{scope}
\draw (1,{0.1+\y*1.4}) -- (4,\y);
\draw (1,{-0.2+\y*1.4}) -- (4,\y);
}
\foreach \x in {1.2,0.6,...,0}{
\draw (4,\x) -- (8,\x);
\draw[fill] (8,\x) circle (0.07cm);
}
\foreach \x in {1.2,0.6,...,-1.2}{
\draw[fill] (4,\x) circle (0.07cm);
}
\draw (2,-1.7) node[below] {$T_{x-1}$};
\draw (6,-0.3) node[below] {$T_x$};
\end{tikzpicture}
\end{center}
\caption{Embedding $T$ during Lemma~\ref{constrembedlem}.}
\end{figure}

\newtheorem{constrembedlem}[thmctr]{Lemma}
\begin{constrembedlem} \label{constrembedlem}
If $A_1, \ldots, A_r \subseteq P$ with $\left|A_i\right| \geq 2^r\beta z$ and $T$ is
a tree on vertex set $[r]$, then
there exist $p_1 \in A_1, \ldots, p_r \in A_r$ such that if $ij \in E(T)$ then $d(p_i, p_j) \geq 2 - \theta$.
\end{constrembedlem}

\begin{proof}
Assume $\left| A_i \right| = 2^r \beta z$.
Let $G$ be the auxiliary graph on vertex set $A_1 \dot{\cup} \dots \dot{\cup} A_r$ where $p \in A_i$ and $q \in A_j$ are adjacent if
$i \neq j$ and $d(p,q) \geq 2 - \theta$.  We would like to find an embedding of $T$ into $G$ such that
$i \in V(T)$ is embedded into $A_i$.  

Let $T_1 \subseteq T_2 \subseteq \dots \subseteq T_{r-1} = T$ be subtrees of $T$ where $T_x$ is formed by deleting
a leaf of $T_{x+1}$.  We prove by induction on $x$ that we can find $2^{r-x} \beta z$ vertex disjoint embeddings
of $T_x$ into $G$ where $i \in V(T_x)$ is embedded into $A_i$.
Since $T_1$ is just a single edge, Lemma~\ref{constrmatchinglem} shows that we can find 
$\left| A_i \right|/2 = 2^{r-1} \beta z$ vertex disjoint embeddings of $T_1$.

Assume $x \geq 2$.  By induction, we can find at least $2^{r-x+1} \beta z$ vertex disjoint embeddings of $T_{x-1}$ into $G$.
Let $j \in V(T_x)$ be the leaf of $T_x$ deleted to form $T_{x-1}$ and let $k$ be the neighbor of $j$ in $T_x$.
Let $A_k'$ be the set of vertices in $A_k$ used by the embeddings of $T_{x-1}$,
so that $\left| A_k' \right| \geq 2^{r-x+1} \beta z$.
We now apply Lemma~\ref{constrmatchinglem} to $A_k'$ and $A_j$ to find a matching between
$A_k'$ and $A_j$ using at least $\left| A_k' \right|/2 \geq 2^{r-x} \beta z$ edges.  Since the vertices of this
matching are distinct, at least $2^{r-x} \beta z$ of the embeddings of $T_{x-1}$ extend to embeddings of $T_x$.
\end{proof}

\newtheorem{constrsmallind}[thmctr]{Lemma}
\begin{constrsmallind} \label{constrsmallind}
For every $s$, $\alpha(\mathcal{H}[V_s]) \leq r^u2^{u+r} \beta z^u$.
\end{constrsmallind}

\begin{proof}
 Fix an arbitrary set $X \subseteq V_s$ with $\left| X \right| = r^u2^{u+r}\beta z^u$.
Let $T_1, \ldots, T_u$ be trees for which $V(T_i) = [r]$ and $\cup T_i$ is the complete graph on vertex set
$[r]$. Observe that the only property that we use about our trees is that they cover the edge set of a $K_r$.
Note that if $\vec{v}_1, \ldots, \vec{v}_r \in X$ such that $d(v_i^{(j)}, v_i^{(\ell)}) \geq 2 - \theta$ when
$j\ell \in E(T_i)$, then $\left\{ \vec{v}_1, \ldots, \vec{v}_r \right\}$ forms a hyperedge inside $X$.
We will find these vertices by repeatedly applying Lemma~\ref{constrembedlem}.

Let $0 \leq j < u$.
Assume we have already selected $v_1^{(1)}, \ldots, v_1^{(j)}, v_2^{(1)}, \ldots, v_2^{(j)}, \linebreak[1] \ldots, v_r^{(1)}, \ldots, \linebreak[1] v_r^{(j)}$, that is coordinates
$1$ through $j$ for all $r$ vertices to be found.
For each $i$, define the set of candidates to continue the future vertex $\vec{v}_i$ as
\begin{align*}
 C_i^{(j)} = \left\{ \left<v_i^{(1)}, \ldots, v_i^{(j)}, q_{j+1}, \ldots, q_u\right> \in X : q_{j+1}, \ldots, q_u \in P \right\}.
\end{align*}
Initially, $C_i^{(0)} = X$.
Throughout the selection process we maintain that the size of $\left| C_i^{(j)} \right|$ is at least
$r^{u-j}2^{u-j+r} \beta z^{u-j}$.

We now show how to select $v_1^{(j+1)}, \ldots, v_r^{(j+1)}$.
For $1 \leq i \leq r$, call a $(j+1)$-tuple $(v_i^{(1)}, \ldots, v_i^{(j)}, p)$ {\bf bad} if
\begin{align*}
\left| \left\{ \left<v_i^{(1)}, \ldots, v_i^{(j)}, p, q_{j+2}, \ldots, q_u \right> \in C_i^{(j)} : q_{j+2}, \ldots, q_u \in P \right\} \right|
 \\ < r^{u-j-1} 2^{u-j-1 + r} \beta z^{u-j-1}.
\end{align*}
Form $D_i$ by deleting all vertices $\vec{w}$ from $C_i^{(j)}$ where the first $j+1$ coordinates of $\vec{w}$ form a bad tuple.
Counting the number of vertices we delete, there are $r$ choices for $i$, there are at most $z$ choices for $p$, and there
are at most $r^{u-j-1} 2^{u-j-1+r} \beta z^{u-j-1}$ choices for the rest of the coordinates.
Thus the number of vertices we delete is at most $r^{u-j}2^{u-j-1+r} \beta z^{u-j}$ so $\left| D_i \right| \geq r^{u-j}2^{u-j-1+r} \beta z^{u-j}$.

Now define
\begin{align*}
A_i = \left\{ p \in P : \exists q_{j+2}, \ldots, q_u \in P \text{ where } \left<v_i^{(1)}, \ldots, v_i^{(j)}, p, q_{j+2}, \ldots, q_u\right> \in D_i \right\}.
\end{align*}
If $\left| A_i \right| < 2^r \beta z$, then $\left| D_i \right| < 2^r \beta z^{u-j} \leq 2^{u-j+r} \beta z^{u-j}$ which is a contradiction.
Now apply Lemma~\ref{constrembedlem} to $A_1, \ldots, A_r$ and $T_{j+1}$ to obtain $v_1^{(j+1)} \in A_1, \ldots, v_r^{(j+1)} \in A_r$.
Since none of the tuples $(v_i^{(1)}, \ldots, \linebreak[1] v_i^{(j+1)})$ are bad,
\begin{align*}
\left| C_i^{(j+1)} \right| \geq r^{u-j-1}2^{u-j-1+r} \beta z^{u-j-1}
\end{align*}
for every $i$.
\end{proof}

\newtheorem{constrlargeedges}[thmctr]{Lemma}
\begin{constrlargeedges} \label{constrlargeedges}
Let $\mathcal{E} = \left\{ \left\{ \vec{v}_1, \ldots, \vec{v}_r \right\} \in \mathcal{H} : \vec{v}_i \in V_i \right\}$.
Then there exists a constant $c$ depending only on $r$ such that
\begin{align*}
\left| V(\mathcal{H}) \right| \leq r2^{-\binom{u}{2}} z^u
\end{align*}
and
\begin{align*}
\left| \mathcal{E} \right| \geq 2^{-\binom{ru}{2}} z^{ru} - c\alpha z^{ru}
  \geq 2^{r\binom{u}{2} - \binom{ru}{2}} \left( \frac{\left| V(\mathcal{H}) \right|}{r} \right)^r -
   c\alpha \left| V(\mathcal{H}) \right|^r.
\end{align*}
\end{constrlargeedges}

\begin{proof}
By Property~\propsizemulti, each $V_i$ has size at most $z \prod_{i=1}^{u-1} \left( \frac{z}{2^i} - i \alpha z \right)$
so the number of vertices is at most $r2^{-\binom{u}{2}} z^u$.
Using Property~\propsizemulti there are at least
\begin{align*}
z \prod_{i=1}^{ru-1} \left( \frac{z}{2^i} -i\alpha z \right)
\end{align*}
choices of $ru$ points on the sphere with pairwise distance at most $\sqrt{2} - \theta$.
Each of these $ru$-sets of points form a cross-hyperedge.
\end{proof}

\section{Proofs of Theorems \ref{mainind3k5}, \ref{maintkthm}, and \ref{constrblowup}} \label{secTheorems}

We now turn our attention to proving Theorems~\ref{mainind3k5}, \ref{maintkthm}, and \ref{constrblowup}.
Consider the construction
$\mathcal{H}$ from Section~\ref{secConstruction} and assume $\tk{r+2}{r}$ is a subhypergraph.
Lemma~\ref{constrnotk4lem} tells us it is impossible to have two core vertices in two different classes,
so we must have three core vertices in some part.  $\mathcal{H}$ itself may contain a copy of $\tk{3}{r}$ inside
one part, but by using an idea of R\"{o}dl~\cite{beg-rodl85} we are able to eliminate this possibility by blowing up the hypergraph
$\mathcal{H}$.  In~\cite{beg-rodl85}, R\"{o}dl proved a variant of the following theorem for graphs and the special case
when $\mathcal{F}$ is a cycle.

\newtheorem{randblowup}[thmctr]{Theorem}
\begin{randblowup} \label{randblowup}
Let $\mathcal{H}$ be an $r$-uniform hypergraph on $n$ vertices.  Let $0 < \gamma < 1$ and let $\ell$ be a positive integer.
Then there exists a $t = t(\mathcal{H},\ell,\gamma,r)$ and an $r$-uniform hypergraph $\mathcal{G}$ with vertex set 
$V(\mathcal{H}) \times [t]$ with the following properties.
\begin{itemize}
\item For all $\left\{ a_1, \ldots, a_r \right\} \in \mathcal{H}$ and all sets
 $U_i \subseteq \left\{ a_i \right\} \times [t]$ with $\left| U_i \right| \geq \gamma t$ for each $1 \leq i \leq r$,
there exists at least one hyperedge of $\mathcal{G}$ with one vertex in each $U_i$.
\item $\mathcal{G}$ does not contain as a subhypergraph any
 $v$-vertex hypergraph $\mathcal{F}$ with $m$ edges where $v \leq \ell$ and $v + (1+\gamma - r)(m-1) < r$.
\end{itemize}
\end{randblowup}

\begin{proof}
Let $\mathcal{H}'$ be the $t$-blowup of $\mathcal{H}$.  That is, $V(\mathcal{H}') = V(\mathcal{H}) \times [t]$ and
the hyperedges are $\left\{ \left\{ (a_1, i_1), (a_2, i_2), \ldots, (a_r, i_r) \right\} : \left\{ a_1, \ldots, a_r \right\} \in \mathcal{H}, 1 \leq i_1, \ldots, i_r \leq t \right\}$.
Let $\mathcal{H}''$ be a random subhypergraph of $\mathcal{H}'$ where each hyperedge is chosen independently with probability $p = t^{1+\gamma-r}$
(note that $r \geq 2$ and $\gamma$ is small so that $p < 1$).
Let $\mathcal{F}$ be a $v$-vertex hypergraph with $\left| \mathcal{F} \right| = m$ and where
$v + (1+\gamma - r) (m-1) < r$.
The expected number of copies of $\mathcal{F}$ in $\mathcal{H}''$ is bounded by $c_1 t^v p^m = o(p t^r)$ where
$c_1$ is some constant depending only on $\mathcal{H}$ and $\ell$.
We now delete one hyperedge from each copy of $\mathcal{F}$ in $\mathcal{H}''$.
There are at most $2^{\ell^r}$ such hypergraphs $\mathcal{F}$ so
we can make $t$ sufficiently large so that we delete fewer than $\frac{\gamma^r}{2} p t^r$ hyperedges.
$\mathcal{G}$ is the resulting graph which now satisfies the second property.

Now fix a hyperedge $E = \left\{ a_1, \ldots, a_r \right\} \in \mathcal{H}$ and $U_i \subseteq \left\{ a_i \right\} \times [t]$
with $\left| U_i \right| = \gamma t$ for $1 \leq i \leq r$.
We now show that the probability that all blowups of the hyperedge $E$ intersecting all $V_i$ are deleted
is exponentially small.
Before deletion, the expected number of blowups of $E$ where the copy of $a_i$ appears in $U_i$ for each $i$ is $p(\gamma t)^r$.
By Chernoff's Inequality, the probability that there are at most $\frac{1}{2} p (\gamma t)^r$ such blowups of $E$ is bounded
by $e^{-c_2 p t^r}$ where $c_2$ is some constant depending only on $\gamma$.  Since we  delete only 
$\frac{1}{2}p (\gamma t)^r$ hyperedges in total,
the probability that we delete all blowups of $E$ where the copy of $a_i$ appears in $U_i$ for each $i$ is at most $e^{-c_2 p t^r}$.

We now use the union bound to bound the probability that there is some hyperedge 
$E = \left\{ a_1, \ldots, a_r \right\} \in \mathcal{H}$ and some 
$U_i \subseteq \left\{ a_i \right\} \times [t]$ with $\left| U_i \right| = \gamma t$ for $1 \leq i \leq r$
where we deleted all blowups of the edge $E$ where the copy of $a_i$ appears in $U_i$ for each $i$.
This probability is bounded by
\begin{align*}
\left| \mathcal{H} \right| \binom{t}{\gamma t}^r e^{-c_2 pt^r} &\leq
\left| \mathcal{H} \right| \left( \frac{e}{\gamma} \right)^{\gamma r t} e^{ - c_2 p t^r } 
\leq \left| \mathcal{H} \right| e^{c_3 t} e^{-c_4 t^{1+\gamma} }
=  o(e^{-t})
\end{align*}
where $c_3$ and $c_4$ are constants depending only on $\gamma$ and $r$.
\end{proof}

By combining the construction from Section~\ref{secConstruction} and the previous theorem, we  prove Theorem~\ref{constrblowup}.

\begin{proof}[Proof of Theorem~\ref{constrblowup}]
Let $z = N$ and let $\mathcal{H} = \mathcal{H}(r, z, \alpha, \beta)$ be the hypergraph constructed in 
Section~\ref{secConstruction} and $V_1, \ldots, V_r$ the partition of the vertex set of $\mathcal{H}$.
Let $\mathcal{E}_1$ be the set of cross-hyperedges, that is 
$\mathcal{E}_1 = \left\{ \left\{ \vec{v}_1, \ldots, \vec{v}_r \right\} \in \mathcal{H} : \vec{v}_i \in V_i \right\}$
and let $\mathcal{E}_2 = \mathcal{H} - \mathcal{E}_1$ so $\mathcal{E}_2$ is the set of hyperedges which are inside some $V_i$.
Let $\gamma = \beta$ and $\ell = r^3$ and apply Theorem~\ref{randblowup} to $\mathcal{E}_2$ to obtain $\mathcal{E}_2'$ where $V(\mathcal{E}_2') = V(\mathcal{H}) \times [t]$.
Let $\mathcal{G}$ be $\mathcal{E}_2'$ together with all the hyperedges 
\begin{align*}
\left\{ \left\{ (\vec{v}_1,a_1),\ldots,(\vec{v}_r,a_r) \right\} : \left\{ \vec{v}_1, \ldots, \vec{v}_r \right\} \in \mathcal{E}_1, 1 \leq a_i \leq t \right\}.
\end{align*}

Let $m = \left| V_i \right| t \approx 2^{-u(u-1)/2} z^u t$ so that $\mathcal{G}$ has $rm$ vertices, and let $W_i = V_i \times [t]$. 
Now we verify the claimed properties of $\mathcal{G}$.

(i) By Lemma~\ref{constrnotk4lem}, $\mathcal{H}$ contains no hypergraph in $\tkf{4}{r}$ embedded so that
$V_i$ has two core vertices and $V_j$ has two core vertices. Since the blow up preserves this, the same holds for  
$\mathcal{G}$, $W_i$, and $W_j$.

(ii) By Lemma~\ref{constrlargeedges}, $\left| \mathcal{E}_1 \right| \geq 2^{-\binom{ru}{2}} z^{ru} - c_2\alpha z^{ru}$.
Because during blow up we keep all cross hyperedges, 
$$e(\mathcal{G})\ge 2^{-\binom{ru}{2}} z^{ru} t^r - c_2\alpha z^{ru} t^r=\left(2^{r\binom{u}{2}-\binom{ru}{2}}-  c_2\alpha 2^{u(u-1)r/2}\right)m^r $$  where $c_2$
is some constant depending only on $r$.
  
(iii) Theorem~\ref{randblowup} shows that $\mathcal{G}[W_i]$ does not contain as a subhypergraph any hypergraph
$\mathcal{F}$ with $\left| V(\mathcal{F}) \right| \leq r^3 = \ell$ and $\left| V(\mathcal{F}) \right| + (1-r)(\left| \mathcal{F} \right|-1) < r$.

(iv) Let $I$ be a vertex set in $\mathcal{G}[W_1]$ with $\left| I \right| = r^u 2^{u+r+1} \beta z^u t$.
For $\vec{v} \in V_1$, call $\vec{v}$ \textbf{$\gamma$-bad} if there are fewer than $\gamma t$ indices $1\leq i \leq t$ such that $(\vec{v},i) \in I$.
Form $I'$ by deleting all pairs $(\vec{v},i)$ from $I$ where $\vec{v}$ is $\gamma$-bad.  We deleted at most $z^u \gamma t = z^u \beta t$ pairs so $\left| I' \right| \geq r^u 2^{u+r} \beta z^u t$.
Define $A = \left\{ \vec{v} \in V_1 : (\vec{v},i) \in I' \text{ for some } 1 \leq i \leq t \right\}$.  Then $\left| A \right| \geq r^u 2^{u+r} \beta z^u$, so by Lemma~\ref{constrsmallind} we must
have a hyperedge $\left\{ \vec{v}_1, \ldots, \vec{v}_r \right\}$ contained in $\mathcal{H}[A] \subseteq \mathcal{E}_2$.
Define $B_i = (\left\{ \vec{v}_i \right\} \times [t]) \cap I'$.  Since no $\vec{v}_i$ is  $\gamma$-bad we have $\left| B_i \right| \geq \gamma t$
for every $1 \leq i \leq r$.
By Theorem~\ref{randblowup} there exists a hyperedge of $\mathcal{E}_2'[W_1] \subseteq \mathcal{G}[W_1]$ with one vertex in each $B_i$, which is a hyperedge contained in $I$.
This shows that the independence number of $\mathcal{G}[W_i]$
 is at most $r^{u}2^{u+r+1} \beta z^u t$ for each $1 \leq i \leq r$, which implies that
$\mathcal{G}$ has independence number at most $r^{u+1}2^{u+r+1} \beta z^u t \leq c_1 \beta m$,
 where $c_1$ is a constant depending only on $r$.
\end{proof}

\begin{proof}[Proof of Theorem~\refmaintkthm]
Let $\mathcal{G}$ be the construction from Theorem~\ref{constrblowup} and assume that $\tk{r+2}{r}$ is a subhypergraph.
Since we cannot have two core vertices in two different parts, the copy of $\tk{r+2}{r}$ must have
three core vertices in one part.  Let $\mathcal{F} = \tk{3}{r}$.  Then $\left| \mathcal{F} \right| = 3$ and
$\left| V(\mathcal{F}) \right| = 3 + 3(r-2) = 3r - 3 < r + 2(r-1) = r + (r-1)(\left| \mathcal{F} \right| - 1) = 3r-2$ which contradicts Theorem~\ref{constrblowup} (iii).

Let $n = \left| V(\mathcal{G}) \right|$.
From Theorem~\ref{constrblowup}, we know that $n = rm$ and that
\begin{align*}
\left| \mathcal{G} \right| \geq 2^{r\binom{u}{2} - \binom{ru}{2}} m^r - c_1 \alpha m^r = 
   2^{r\binom{u}{2} - \binom{ru}{2}} \left( \frac{n}{r} \right)^r - c_2 \alpha n^r,
\end{align*}
where $c_1$ and $c_2$ are constants depending only on $r$.  Thus for any $\alpha > 0$, we know that
\begin{align*}
\lim_{\beta \rightarrow 0} \lim_{n \rightarrow \infty} \frac{\rt(n,\tk{r+2}{r}, \beta n)}{n^r}
&\geq \lim_{\beta \rightarrow 0} \lim_{n \rightarrow \infty} \frac{ 2^{r\binom{u}{2}-\binom{ru}{2}} \left(\frac{n}{r}\right)^{r} - c_2\alpha n^{r} }{n^r} \\
&\geq 2^{r\binom{u}{2} - \binom{ru}{2}} \left( \frac{1}{r} \right)^r - c_2 \alpha,
\end{align*}
yielding  $\theta(\tk{r+2}{r}) \geq 2^{r\binom{u}{2} - \binom{ru}{2}} (1/r)^r.$
\end{proof}

\begin{proof}[Proof of Theorem~\refmaintkfthm]
The proof is similar to the proof of Theorem~\refmaintkthm. No copy of a hypergraph in $\tkf{2r}{r}$ can have two core vertices in two different parts
so it must have at least $r+1$ core vertices in a single part.
To complete the proof, we just need to show that every minimal hypergraph $\mathcal{F}$ in $\tkf{r+1}{r}$
satisfies $\left| V(\mathcal{F}) \right| \leq r + (r-1)(\left| \mathcal{F} \right|-1)$.  Let $v_1, \ldots, v_{r+1}$ be the core vertices of $\mathcal{F}$.
For $1 \leq a < b \leq r+1$, since $\mathcal{F} \in \tkf{r+1}{r}$
there exists some hyperedge containing both $\vec{v}_a$ and $\vec{v}_b$.
Let $E_{a,b}$ be a hyperedge of containing both $\vec{v}_a$ and $\vec{v}_b$ (if there are  more than
one such hyperedges, pick one arbitrarily.)
Now consider the ordering 
\begin{align*}
E_{1,2}, E_{1,3}, \ldots, E_{1,r+1}, E_{2,3}, \ldots, E_{2,r+1}, E_{3,4}, \ldots, E_{r,r+1}.
\end{align*}
Since $\mathcal{F}$ is minimal, all hyperedges of $\mathcal{F}$ appear
in the ordering somewhere.  Now let $F_1, \ldots, F_m$ be a list of the hyperedges of $\mathcal{F}$ where for each hyperedge $D \in \mathcal{F}$,
we keep the first copy of $D$ in the ordering and remove all other copies.
By the choice of ordering, each $F_i$ must use at least one vertex from the previous hyperedges.
Therefore, $\left| V(\mathcal{F}) \right| \leq r + (r-1)(m-1)$.  In fact,
the last hyperedge must use at least two previous vertices so we can reduce the bound by one to $\left| V(\mathcal{F}) \right| \leq r + (r-1)(m-1) - 1$.  
\end{proof}

The \textbf{shadow graph} of a hypergraph $\mathcal{H}$ is a graph $G$ with $V(G) = V(\mathcal{H})$ and $xy \in E(G)$  if there exists some hyperedge $E$ of $\mathcal{H}$
with $x,y \in E$.
We will now show that Theorem~\ref{mainind3k5} follows by looking at the shadow graph of the hypergraph from Theorem~\ref{constrblowup}.

\begin{proof}[Proof of Theorem~\ref{mainind3k5}]
Let $r = t$ and let $\mathcal{G}$ be the hypergraph constructed in Theorem~\ref{constrblowup} with parts
$W_1, \ldots, W_r$.  Let $G$ be the shadow graph of $\mathcal{G}[W_1 \cup W_2 \cup \ldots \cup W_{\ell}]$, so
we  take only  the shadow graph of the first $\ell$ parts.  Then $\alpha_t(G)$ is small because any hyperedge
inside $\mathcal{G}[W_i]$ turns into a copy of $K_t$ in $G[W_i]$ for $1 \leq i \leq \ell$.
Assume $G$ contains $K_{t+\ell}$.  It is not possible to have two of the vertices in $W_i$ and two of the vertices in $W_j$ with $i \neq j$ because then $\mathcal{G}$ would contain a $\tkf{4}{r}$ arranged
so that two core vertices are in $W_i$ and two core vertices are in $W_j$.  
Thus we can assume without loss of generality that $G[W_1]$ contains $K_{t+1}$.  This implies
that $\mathcal{G}[W_1]$ contains a hypergraph in $\tkf{r+1}{r}$ which was excluded by the proof of 
Theorem~\refmaintkfthm.

To compute the number of edges of $G$, we must use Property~\propsize.
Edges between $W_i$ and $W_j$ are chosen by picking $2u$ points within distance $\sqrt{2} - \theta$ on the sphere
and then blowing each vertex up into size $t$.  Therefore, we have at least
$t^2 z \prod_{i = 1}^{2u - 1} \left( \frac{z}{2^i} - i\alpha z \right)$ edges between $W_i$ and $W_j$.  Thus
\begin{align} \label{mainedgelowereq}
\left| E(G) \right| \geq \binom{\ell}{2} 2^{-\binom{2u}{2}} z^{2u} t^2 - c_1\alpha z^{2u}t^2,
\end{align}
where $c_1$ is some constant depending only on $r$.
Each $W_i$ has size at most $2^{-\binom{u}{2}} z^u t$ so $G$ has at most $\ell 2^{-\binom{u}{2}} z^u t$ vertices.
Thus
\begin{align} \label{mainvertexuppereq}
\frac{2^{\binom{u}{2}}}{\ell} \left| V(G) \right| &\leq z^u t.
\end{align}
Combining \eqref{mainedgelowereq} with \eqref{mainvertexuppereq}, we obtain
\begin{align*}
\left| E(G) \right| &\geq \binom{\ell}{2} 2^{-\binom{2u}{2}} \left( \frac{2^{\binom{u}{2}}}{\ell}
                            \left| V(G) \right| \right)^2 - c_2\alpha \left| V(G) \right|^2 \\
                    &\geq \frac{1}{2} \frac{\ell(\ell - 1)}{\ell^2} 2^{u(u-1) - u(2u-1)} \left| V(G) \right|^2 
                            - c_2\alpha \left| V(G) \right|^2 \\
                    &\geq \frac{1}{2} \left( 1 - \frac{1}{\ell} \right) 2^{-u^2} \left| V(G) \right|^2 
                            - c_2\alpha \left| V(G) \right|^2
\end{align*}
for some constant $c_2$ depending only on $r$.
\end{proof}

\section{Lower bounds on the  Ramsey-Tur\'{a}n threshold functions} \label{secSmallInd}

The main tool to prove Theorem~\ref{k3tkthm} is the method of dependent random choice.
It is a simple yet surprisingly powerful technique which has found applications in
Extremal Graph Theory, Ramsey Theory, Additive Combinatorics, and Combinatorial Geometry.  
Early versions of this technique were proved and applied by several researchers,
starting with Gowers, Kostochka, R\"{o}dl, and Sudakov.
Gowers \cite{drc-gowers98} used a variant of dependent random choice in an alternate proof of Szemer\'{e}di's Theorem \cite{sze-szemeredi75}
for four-term arithmetic progressions, Kostochka and R\"{o}dl \cite{drc-kostochka01} used it to investigate bipartite Ramsey numbers,
and Sudakov \cite{rt-sudakov03} used it to prove $\rt(n,K_4,2^{-w(n) \sqrt{\log n}}) = o(n^2)$, where $w(n)$ is arbitrary function tending to infinity.
Since then, many other applications of the dependent random choice method have been found (see~\cite{drc-fox} for a survey).

\newtheorem{deprand}[thmctr]{Lemma}
\begin{deprand} \label{deprand}
(Dependent Random Choice, Lemma 2.1 in~\cite{drc-fox}).
Let $a,m,n,r,t$ be positive integers.  Let $G$ be an $n$-vertex graph with average
degree $d := 2\left| E(G) \right|/n$.  If
\begin{align*}
\frac{d^t}{n^{t-1}} - \binom{n}{r} \left( \frac{m}{n} \right)^t \geq a,
\end{align*}
then $G$ contains a subset $U$ of at least $a$ vertices such that any $r$ vertices in $U$ have at least $m$
common neighbors.
\end{deprand}

Conlon, Fox, and Sudakov~\cite{drc-conlon09}, investigating  
the Ramsey numbers of sparse hypergraphs,
extended Lemma~\ref{deprand} to hypergraphs.  The \textbf{weight} $w(S)$ of a set $S$ of hyperedges in
a hypergraph is the number of vertices in the union of these edges.

\newtheorem{hypdeprand}[thmctr]{Lemma}
\begin{hypdeprand} \label{hypdeprand}
(Hypergraph Dependent Random Choice, Lemma 1 in~\cite{drc-conlon09}).
Suppose $s, \Delta$ are positive integers, $\epsilon, \beta > 0$, and
$G_r$ is an $r$-uniform, $r$-partite hypergraph with parts $V_1, \ldots, V_r$, each part having size $N$.  Suppose
$G_r$ has at least $\epsilon N^r$ edges.  Then there exists an $(r-1)$-uniform, $(r-1)$-partite hypergraph $G_{r-1}$
on the vertex set $V_2 \cup \ldots \cup V_r$ which has at least $\frac{1}{2} \epsilon^s N^{r-1}$ edges and such that for each
nonnegative integer $w \leq (r-1)\Delta$, there are at most $4 r \Delta \epsilon^{-s} \beta^s w^{r\Delta} r^w N^w$
dangerous sets of edges of $G_{r-1}$ with weight $w$, where a set $S$ of edges of $G_{r-1}$ is dangerous if
$\left| S \right| \leq \Delta$ and the number of vertices $v \in V_1$ such that for every edge $e \in S$,
$e + v \in G_r$ is less than $\beta N$.
\end{hypdeprand}

The main idea of the proof of Theorem~\ref{k3tkthm} is to first apply Lemma~\ref{hypdeprand}
to obtain a graph $G$ and then apply
Lemma~\ref{deprand} to $G$.  Lemma~\ref{deprand} guarantees a set $U$  large enough so that
we can find a hyperedge $E_3$ contained inside $U$.
The vertices of $E_3$  have a large number of common neighborhood in $G$, sufficient  to find a hyperedge $E_2$ among the common neighbors.
Then the hypergraph dependent random choice lemma shows that we can extend the edges of $G$ spanned by $E_2 \cup E_3$ to hyperedges.  We  thus  find the following
hypergraph.
Let $F$ be the $3$-uniform hypergraph with vertices $\left\{ x_1,x_2,x_3,y_1,y_2,y_3,z_1,z_2,z_3 \right\}$ and edges 
$\left\{ x_1x_2x_3,y_1y_2y_3,z_1z_2z_3 \right\} \cup \{ x_iy_jz_k : 1 \leq \linebreak[1] i,j,k \leq 3 \}$.
Note that $F \in \tkf{9}{3}$.
For a 3-uniform hypergraph, the \textbf{codegree} $d(x,y)$ of a pair of vertices $x,y$ is the number of edges $E$ with $x,y \in E$.

\newtheorem{tk6}[thmctr]{Theorem}
\begin{tk6} \label{tk6}
Let $\gamma = \gamma(n)$ be any function going to infinity arbitrarily slowly.
Let $\beta = \beta(n) = 2^{-\gamma (\log n)^{2/3}}$.
There exists a constant $b$ such that if $\mathcal{H}$ is an $n$-vertex, $3$-uniform hypergraph with independence number at most
$\frac{1}{3}\beta n$ and at least $b n^3 2^{-\gamma^3/28} + 144 n^2 = o(n^3)$ edges, then $\mathcal{H}$ contains $F$ and $\tk{6}{3}$.
\end{tk6}

\begin{proof}
Let $N = \frac{n}{3}, \Delta = 9, w = 6, r = 3, c = 4 r \Delta w^{r \Delta} r^w,
s = \frac{w+1}{\gamma} \sqrt[3]{\log n}$, and $\epsilon = 2^{-\gamma^2 \sqrt[3]{\log n}/4}$.  Let $b = 9 c$, so $\mathcal{H}$ has at least
$9 c \epsilon^{1/s} N^3 + 144 n^2$ edges and independence number at most $\frac{1}{3}\beta n$.

For simplicity, assume $3$ divides $n$ and let $\mathcal{H}'$ be a $3$-partite subhypergraph of $\mathcal{H}$ with equal part sizes.  We can always find such a hypergraph $\mathcal{H}'$ with at least $\frac{1}{9}$
of the edges of $\mathcal{H}$.  For each pair $x,y$ of vertices in different parts, delete
all edges containing both $x$ and $y$ if the codegree $d(x,y)$
is at most $16$.  We delete at most $16n^2$ hyperedges.
Thus we have a $3$-partite hypergraph $\mathcal{H}'$ with at least $c \epsilon^{1/s} N^3$ edges and the
codegree of any pair of vertices from different parts is zero or at least $16$.
Let $V_1, V_2, V_3$ be the parts of $\mathcal{H}'$, each part having size $N$.

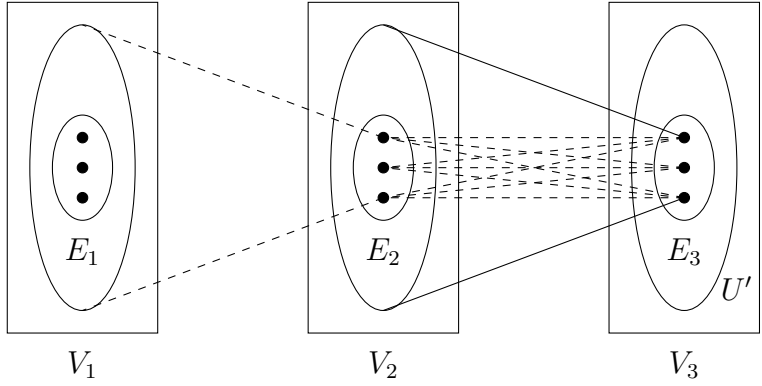
\begin{figure}
\begin{center}
\begin{tikzpicture}
\draw (-1,-2.2) rectangle (1,2.2);
\draw (0,-2.3) node[below] {$V_1$};
\draw (0,0) ellipse (0.7 and 1.9);
\draw (0,0) ellipse (0.4 and 0.7);
\draw (0,-0.8) node[below] {$E_1$};
\draw[fill] (0,0.4) circle (0.07cm);
\draw[fill] (0,0) circle (0.07cm);
\draw[fill] (0,-0.4) circle (0.07cm);
\begin{scope}[xshift=4cm]
\draw (-1,-2.2) rectangle (1,2.2);
\draw (0,-2.3) node[below] {$V_2$};
\draw (0,0) ellipse (0.7 and 1.9);
\draw (0,0) ellipse (0.4 and 0.7);
\draw (0,-0.8) node[below] {$E_2$};
\draw[fill] (0,0.4) circle (0.07cm);
\draw[fill] (0,0) circle (0.07cm);
\draw[fill] (0,-0.4) circle (0.07cm);
\end{scope}
\begin{scope}[xshift=8cm]
\draw (-1,-2.2) rectangle (1,2.2);
\draw (0,-2.3) node[below] {$V_3$};
\draw (0,0) ellipse (0.7 and 1.9);
\draw (0.7,-1.6) node {$U'$};
\draw (0,0) ellipse (0.4 and 0.7);
\draw (0,-0.8) node[below] {$E_3$};
\draw[fill] (0,0.4) circle (0.07cm);
\draw[fill] (0,0) circle (0.07cm);
\draw[fill] (0,-0.4) circle (0.07cm);
\end{scope}
\draw (8,0.4) -- (4,1.9);
\draw (8,-0.4) -- (4,-1.9);
\foreach \a in {-0.4,0,0.4}{
  \foreach \b in {-0.4,0,0.4}{
    \draw[dashed] (8,\a) -- (4,\b);
  }
}
\draw[dashed] (4,0.4) -- (0,1.9);
\draw[dashed] (4,-0.4) -- (0,-1.9);
\end{tikzpicture}
\caption{Embedding $F$ in Theorem~\ref{tk6}.}
\end{center}
\end{figure}

We now apply Lemma~\ref{hypdeprand} to $\mathcal{H}'$ to obtain a graph $G$ on $V_2 \cup V_3$
with at least $\frac{1}{2} \left(c \epsilon^{1/s}\right)^s N^2 \geq 2\epsilon N^2$
edges and at most 
\begin{align*}
4r\Delta\left(c\epsilon^{1/s}\right)^{-s}\beta^s w^{r\Delta}r^w N^w &\leq \epsilon^{-1} \beta^s N^6
\leq 2^{\gamma^2 \sqrt[3]{\log n}/4 - 7 \log n} N^6
\end{align*}
dangerous sets of edges of weight 6. When $n$ is large, the number of dangerous sets is at most $1/2$ so we can assume $G$ has no dangerous sets of weight 6.

We now apply Lemma~\ref{deprand} to $G$ with $t = \frac{4}{\gamma} \sqrt[3]{\log n}$,
$d = 4\epsilon N$, and $a = m = 2\beta N$.  Let $n_1 = \left| V(G) \right| = 2N$.  We check
\begin{align*}
  \frac{d^t}{n_1^{t-1}} - \binom{n_1}{3} \left(\frac{m}{n_1}\right)^t &\geq 4 \epsilon^t N - 2N^3 \beta^t
  \geq 2^{2-\gamma (\log n)^{2/3}} N - 2^{1-4 \log n} N^3 \\
  &\geq 4\beta N - \frac{1}{2} 
  \geq a = m.
\end{align*}
Therefore we have a subset $U$ of $V(G)$ with $\left| U \right| = m = 2\beta N$ such that every three vertices of U
have at least $\beta N$ common neighbors in $G$.  Either $V_2$ or $V_3$ contains at least half of the vertices of $U$,
so assume by symmetry that $U' = U \cap V_3$ has at least $\beta N$ vertices.

The set $U'$ contains a hyperedge $E_3$ of $\mathcal{H}$ since the size of $U'$ is larger than the independence number of $\mathcal{H}$.
The vertices of $E_3$ have at least $\beta N$ common neighbors in $G$, so the common neighbors contain a hyperedge $E_2$.
 By Lemma~\ref{hypdeprand} $G$ is bipartite, so $E_3 \subseteq V_3$ implies that $E_2 \subseteq V_2$.
If we take $S$ to be the nine edges of $G$ spanned by the vertices $E_2 \cup E_3$, then $S$ has weight $6$ so it is not dangerous.
Therefore, we find at least $\beta N$ vertices $v$ in $V_1$
such that $v x y$ is a hyperedge for all $x \in E_2$ and all $y \in E_3$.  These $\beta N$ vertices contain a hyperedge $E_1$ of $\mathcal{H}$.

Let $E_1 = \left\{ x_1,x_2,x_3 \right\}$, $E_2 = \left\{ y_1,y_2,y_3 \right\}$,
and $E_3 = \left\{ z_1,z_2,z_3 \right\}$. These vertices form a copy of $F$ within $\mathcal{H}$.
We also find a copy of $\tk{6}{3}$ with core vertices $x_1,x_2,y_1,y_2,z_1,z_2$.
Vertices $x_1$ and $ x_2$ are contained together in the hyperedge $x_1x_2x_3$.  Since $x_i$ and $y_j$ are contained together in at least one hyperedge of $\mathcal{H}'$,
the codegree of $x_i$ and $y_j$ in $\mathcal{H}$ is at least $16$.  We can therefore find a distinct vertex in $V_3$
which is contained in a hyperedge together with $x_i$ and $y_j$.  The pairs $x_i,z_j$ and $y_i,z_j$ are handled similarly.
\end{proof}

\section{Open problems} \label{secOpen}

There are many open problems remaining in Ramsey-Tur\'{a}n theory.

\begin{itemize}
\item  The exact value of $\rt_3(n,K_s,o(n))$ for small values of $s$ are mostly still unknown.  
Erd\H{o}s, Hajnal, Simonovits, S\'{o}s, and Szemer\'{e}di  \cite{rt-erdos94} proved that $\theta_3(K_s) = \frac{1}{2} \left(1-\frac{3}{s-1}\right) \linebreak[1] $
when $s \equiv 1 \pmod 3$.  
The best bound for $s=5$ is our lower bound of $\frac{1}{64} $ and an upper bound of $\frac{1}{12} n^2$ by 
\cite{rt-erdos94}. For $s = 6$, $\frac{1}{48}\le \theta_3(K_6)  \le \frac{1}{6}$.  In \cite{rt-erdos94}
a construction is given which is conjectured to show $\theta_3(K_6) \geq 1/8$; most likely the construction is correct.
Based on these bounds, the following question is natural. Is there a construction determining $\theta(K_s)$ where the
density of edges between classes is not $2^{-\ell}$ for some integer $\ell$?
 
\item  In the area of the Ramsey-Tur\'{a}n theory, one of the major open problems is to prove a generalization of the Erd\H{o}s-Stone Theorem \cite{rrl-erdos46} by proving that $\theta(H) = \theta(K_s)$
where $s$ is equal to some parameter depending only on $H$.
Let $s$ be the minimum number such that $V(H)$ can be partitioned into $\left\lceil s/2 \right\rceil$ sets $V_1, \ldots, V_{\left\lceil s/2 \right\rceil}$
such that $V_1, \ldots, V_{\left\lfloor s/2 \right\rfloor}$ span forests and if $s$ is odd $V_{\left\lceil s/2 \right\rceil}$ spans an independent set.
In \cite{rt-erdos94} it was proved that  $\theta(H) \leq \theta(K_s)$, where the inequality is sharp for odd $s$.
In several papers, Erd\H{o}s mentioned the  simplest open case when $H = K_{2,2,2}$, where one would like to  know at least if $\theta(K_{2,2,2}) = 0$
(see~\cite[Problem 4]{rt-simonovits01},~\cite[p. 72]{rt-erdos83}, \cite[Problem 1.3]{rt-sudakov03} among others).

\item Can the theorem of Bollob\'as~\cite{beg-bollobas89} can be (partially) saved?  We think that the following version of the Erd\H os Conjecture could be true.
Recall that for $A \subseteq \mathbb{S}^k$ and $t \geq 2$, 
\begin{align*}
d_t(A) = \sup \left\{ \min_{i \neq j} d(x_i, x_j) : x_1, \ldots, x_t \in A \right\}.
\end{align*}

\newtheorem{erdconj}[thmctr]{Conjecture}
\begin{erdconj}\label{erdbol}
For every $t$   positive integer and $\epsilon>0$ there is a $k_0$ such that the following holds: 
For every  $k>k_0$ and measurable $A \subseteq \mathbb{S}^k$   if $C\subseteq \mathbb{S}^k$ is a spherical cap with $\mu(C) = \mu(A)>\epsilon$, then $d_t(A) \geq d_t(C)$.
\end{erdconj}

\end{itemize}

\textbf{Acknowledgement.}  The authors would like to thank Dhruv Mubayi for suggesting the study of the 
Ramsey-Tur\'{a}n numbers of $\tk{s}{r}$.  We are also indebted to
B\'{e}la Bollob\'{a}s, Jane Butterfield, Imre Leader, and Wojciech Samotij for helpful discussion and feedback.

\bibliographystyle{amsplain}
\bibliography{refs}

\end{document}